\documentclass[11pt]{article}

\usepackage{latexsym}
\usepackage{amssymb}
\usepackage{amsthm}
\usepackage{amscd}
\usepackage{amsmath}
\usepackage{verbatim}
\usepackage{xcolor}
\usepackage{enumerate}
\usepackage{graphicx}

\newtheorem{theorem}{Theorem}[section]

\newtheorem{lemma}[theorem]{Lemma}
\newtheorem*{lemma*}{Lemma}

\newtheorem*{definition*}{Definition}
\newtheorem{definition}[theorem]{Definition}

\newtheorem{corollary}[theorem]{Corollary}

\theoremstyle{definition}

\numberwithin{equation}{section}

\let\inf\relax \DeclareMathOperator*\inf{\vphantom{p}inf}

\newcommand{\leqn}{\begin{equation}\label}
\def\endeqn{\end{equation}}

      \def\RR{\mathbb{R}}

\newcommand{\res}{\hbox{{\vrule height .22cm}{\leaders\hrule\hskip .2cm}}}

\makeatletter
\renewcommand{\@makefnmark}{\mbox{\textsuperscript{}}}
\makeatother

\def\adots{\mathinner{\mkern2mu\raise0pt\hbox{.}  
\mkern2mu\raise4pt\hbox{.}\mkern1mu
\raise7pt\vbox{\kern7pt\hbox{.}}\mkern1mu}}
\def\res{\hbox{ {\vrule height .22cm}{\leaders\hrule\hskip.2cm} } }

\allowdisplaybreaks[1]

\begin{document}
\title{\bf{Uniformly distributed measures have Big Pieces of Lipschitz Graphs locally}}
\author{ A. Dali Nimer}
\date{}

\maketitle\footnote{The author was partially supported by NSF RTG 0838212, DMS-1361823 and DMS-0856687}
\footnote{Department of Mathematics, University of Chicago, 5734, S. University Ave., Chicago, IL, 60637

E-mail address: nimer@uchicago.edu}
\footnote{Mathematics Subject Classification Primary 28A33, 49Q15}
\begin{abstract}
The study of uniformly distributed measures was crucial in Preiss' proof of his theorem on rectifiability of measures with positive density. It is known that the support of a uniformly distributed measure is an analytic variety. In this paper, we provide quantitative information on the rectifiability of this variety. Tolsa had already shown that $n$-uniform measures have big pieces of Lipschitz graph(BPLG). Here, we prove that a uniformly distributed measure has BPLG locally.
\end{abstract}
\section{Introduction}
Understanding the geometry of uniformly distributed measures has been an important question in geometric measure theory ever since Preiss proved his remarkable theorem on the $n$-rectifiability of measures in $\cite{P}$. This theorem states that, given a Radon measure $\sigma$ in $\mathbb{R}^{d}$, if the $n$-density of $\sigma$ 
\begin{equation}\nonumber \Theta^{n}(x,\sigma)=\lim_{r \to 0} \frac{\sigma(B(x,r))}{r^n}  \end{equation}
exists, is finite and positive $\sigma$-almost everywhere on $\mathbb{R}^{d}$, then there exists a countably $n$-rectifiable set $E$ such that $\sigma(\mathbb{R}^d \backslash E)=0$.
The proof of Preiss' theorem relied heavily on the study of uniformly distributed measures. Indeed, these measures appear as blow ups (zoom-ins) and blow downs (zoom-outs) of measures with positive finite density.
We say a Radon measure $\mu$ in $\mathbb{R}^{d}$ is uniformly distributed if there exists a positive function $\phi: \mathbb{R}_{+} \to \mathbb{R}_{+}$, called its distribution function, such that:
\begin{equation}\nonumber
\mu(B(x,r))=\phi(r), \mbox{ for all } x \in \mbox{supp }\mu, \mbox{ for all } r>0.
\end{equation}
An example of note is when the function $\phi$ is $c r^{n}$ for some $c>0$, $n\leq d$. These are called $n$-uniform measures and appear in many different contexts from geometric measure theory to harmonic analysis and PDE's (for instance in $\cite{KT}$, $\cite{DKT}$, $\cite{PTT}$).

The geometry of the supports of uniformly distributed measures is not very well understood.
Let us start by stating some known facts.
As a direct consequence of Preiss' theorem, we can deduce that the support of an $n$-uniform measure is countably $n$-rectifiable. In fact, the same can be said of uniformly distributed measures. Indeed, Preiss proved in $\cite{P}$ that uniformly distributed measures ``look like'' $n$-uniform measures on small and on large scales. Their $n$-rectifiability can easily be deduced from that fact. 

One might expect much more regularity than rectifiability, given the fact that the property of being uniformly distributed is a global one (i.e. it is a property for all $r>0$). This turns out to be the case.
For $n$-uniform measures, a classification is available in some cases.
 In $\cite{P}$, Preiss provides a classification of the cases $n=1,2$ in $\mathbb{R}^{d}$ for any $d$ . In these cases, $\mu$ is Hausdorff measure restricted to a line or a plane respectively.
In $\cite{KoP}$, Kowalski and Preiss proved that $\mu$ is ($d-1$)-uniform in $\mathbb{R}^d$ if and only if $\mu=\mathcal{H}^{d-1} \res V$
where $V$ is a ($d-1$)-plane, or $d \geq 4$ and there exists an orthonormal system of coordinates
in which $\mu= \mathcal{H}^{d-1} \res (C \times W)$ where $W$ is a ($d-4$)-plane and $C$ is the KP-cone (Kowalski-Preiss cone)
\begin{equation}\label{KPcone}
C =\{(x_1,x_2,x_3,x_4); x_4^2=x_1^2+x_2^2+x_3^2 \}.
\end{equation}
The classification for $n \geq 3$ and codimension $\geq 2$ remains an open question.

On the other hand, in $\cite{KiP}$, Kirchheim and Preiss proved that the support of a uniformly distributed measure is an analytic variety, that is the intersection of zero sets of analytic functions. More precisely:

\begin{theorem}[\cite{KiP}]\label{analyticvar}
Let $\mu$ be a uniformly distributed measure over $\RR^d$ and let $u \in \Sigma$ where $\Sigma =\text{supp}\mu$. For every $x \in \RR^d$ and $s>0$ let \begin{equation}
F(x,s)=\int_{\RR^d} e^{-s|z-x|^2} - e^{-s|z-u|^2} d\mu(z)
\end{equation}
Then:
\begin{itemize}
\item $F(x,s)$ is well-defined and finite for any $x \in \RR^d$ and any $s >0$; moreover its definition is independent of the choice of $u \in \Sigma$
\item $\Sigma = {\bigcap}_{s>0} \left\lbrace x \in {\RR}^{d} ;F(x,s)=0 \right\rbrace $ 
\end{itemize}
\end{theorem}

It is a known fact that an analytic variety of dimension $n$ is a finite union of analytic $n$-submanifolds up to a set of $\mathcal{H}^{n}$-measure 0. This confirms the expectation of regularity but has the disadvantage of not providing any quantitative information on the regularity of the support.

Let us now turn  to uniform rectifiability. This notion was introduced by David and Semmes (see for example $\cite{DS2}$). 
It is a quantitative version of the notion of $n$-rectifiability. One possible definition of it is the following.

Let $\mu$ be a locally finite Radon measure in $\mathbb{R}^d$, and $\Sigma$ its support, $0 \in \Sigma$.
 We say that an Ahlfors $n$-regular measure$\mu$  has big pieces of Lipschitz graphs (BPLG) if there exist constants $\theta$ and $M$ so that,
 for each $x \in \Sigma$ and $R>0$, there is a Lipschitz function $g$ from $\mathbb{R}^n$ to $\mathbb{R}^{d-n}$ such that $g$ has Lipschitz norm not exceeding $M$ and such that its graph $\Gamma$ (up to rotation) satisfies:
$$\mu(B(x,R) \cap \Gamma)\geq \theta R^n.$$
We say that $\mu$ has BPLG locally if given $K$ a compact set of $\mathbb{R}^{d}$ then for every $x \in \Sigma \cap K$ and every $0<R<diam(K)$, 
there is a Lipschitz function $g$ from $ \mathbb{R}^n$ to $\mathbb{R}^{d-n}$ such that $g$ has Lipschitz norm not exceeding $M$ and such that its graph $\Gamma$ (up to rotation) satisfies:
$$\mu(B(x,R) \cap \Gamma)\geq \theta R^n.$$

In $\cite{T}$, Tolsa proved that $n$-uniform measures have BPLG.

\begin{theorem}\cite{T} \label{Tolsa}
Let $\mu$ be an $n$-uniform measure in $\RR^{d}$. Then $\mu$ has big pieces of Lipschitz graphs.
\end{theorem}

Since uniformly distributed measures ``look like'' $n$-uniform measures on small scales one might expect this result to hold locally for uniformly distributed measures. In this paper, we will prove that this is indeed the case. Namely, we will prove the following theorem:

\begin{theorem}\label{mainunif}
Let $\mu$ be a uniformly distributed measure in $\RR^{d}$. Then there exists $n$ such that $\mu$ has big pieces of Lipschitz graph locally.
\end{theorem}
The proof is analogous to Tolsa's proof of Theorem $\ref{Tolsa}$. To apply the techniques that the author introduced in $\cite{T}$, one needs to use the fact that uniformly distributed measures locally behave like $n$-uniform measures and are radially invariant. These two properties allow us to obtain estimates on the Riesz transforms and to prove that every ball in $\Sigma$ the support of $\mu$ contains a relatively large ball that is flat. 

The second step consists in proving that flatness is stable for uniformly distributed measures. In other words, if the support is flat at small enough scale it will be flat at all smaller scales. The fact that uniformly rectifiable measures have $n$-uniform pseudo-tangents is the key idea allowing us to generalize the stability of flatness for $n$-uniform measures to uniformly distributed measures.

\newpage

\section{Preliminaries}
Let us first define the support of a measure.
\begin{definition}
Let $\mu$ be a measure in $\RR^{d}$. We define the support of $\mu$ to be 
\begin{equation}
\mbox{supp}(\mu)=\left\lbrace x \in \RR^{d} ; \mu(B(x,r))>0, \mbox{ for all } r>0 \right\rbrace. 
\end{equation}
Note that the support of a measure is a closed subset of $\RR^{d}$.
\end{definition}
We start with some facts about uniformly distributed measures.
The first is a theorem by Preiss describing the behavior of uniformly distributed measures at small and large scales.

\begin{theorem}[\cite{P}, Theorem 3.11] \label{dim} Suppose $\mu$ is uniformly distributed in $\mathbb{R}^d$, and let $\phi$ be its distribution function. Then there exist integers $n$ and $p$ such that: $$ \lim_{r \to 0} \frac{\phi(r)}{r^n} \mbox{ and } \lim_{r \to \infty} \frac{\phi(r)}{r^p} \mbox{  both exist.}$$ 
We denote $n$ and $p$ by $$n=dim_0 \mu \mbox{ and } p=dim_{\infty} \mu.$$ \end{theorem}

We can deduce the following useful corollary about the growth of $\mu$ at small scales from this theorem.
\begin{corollary}\label{locallyAR}
Suppose $\mu$ is a uniformly distributed measure with $dim_{0} \mu = n$ and $dim_{\infty} \mu= p$, and $\phi$ the function associated to $\mu$. Let $R \in {\mathbb{R}}_{+}$. There exists $C \in {\mathbb{R}}_{+}$ depending on $R$ such that for all $r \leq R$, the following holds:
\begin{equation} C^{-1} r^n \leq \phi(r) \leq C r^n. 
\end{equation}
\end{corollary} 
\begin{proof}
According to Theorem $\ref{dim}$ there exist $r_0$ and $r_{\infty}$ such that:

\begin{align*} \mu(B(x,r)) \sim r^{n}, & \; x \in \Sigma, r \leq r_0 \\
                               \mu(B(x,r)) \sim r^{p}, &\; x \in \Sigma,  r \geq r_{\infty} 
                               \end{align*}  
If $R \leq r_0$, the statement follows with a $C$ not depending on  
$R$.
First, assume $r_0 \leq R \leq r_{\infty}$ and take $r$ such that $r_0 \leq r \leq R$. Then:
\begin{equation} \frac{{r_0}^n}{R^n} r^n  \lesssim \phi(r_0) \leq \phi(r) \leq \phi(r_{\infty}) \lesssim \frac{{r_{\infty}}^p}{{r_0}^n} r^n 
\end{equation}
Now assume $R \geq r_{\infty}$ and let $r \leq R$.
If $r_0 \leq r \leq r_\infty$, then: 
\begin{equation} 
\frac{{r_0}^n}{{r_{\infty}}^n} r^n \leq {{r_0}^n} \lesssim \phi(r_0) \leq \phi(r) \leq \phi(r_{\infty}) \lesssim{ r_{\infty}}^p \leq \frac{{r_{\infty}}^p}{{r_0}^n} r^n. 
\end{equation}
Finally, suppose $r_{\infty} \leq r \leq R$. Then:

\begin{equation}
\frac{{r_{\infty}}^p}{R^n} r^n \leq {r_{\infty}}^p \lesssim \phi(r) \lesssim {r}^p \leq \frac{R^p}{{r_{\infty}}^n} r^n. \end{equation}

\end{proof}

Another theorem in $\cite{KiP}$ states that uniformly distributed measures don't grow too fast.

\begin{theorem}[\cite{KiP}, Lemma 1.1]\label{doublingKiP}
Let $\mu$ be a uniformly distributed measure over $\RR^{d}$, $x \in \RR^{d}$, $0<s<r<\infty$ and $\phi$ its distribution function. Then $\mu(B(x,r)) \leq 5^{d} \left( \frac{r}{s} \right)^{d} \phi(s)$.

\end{theorem}

Another interesting feature of uniformly distributed measures is that radial functions integrate nicely against them.

\begin{theorem}\label{radial}
Let $\mu$ be a uniformly distributed measure on $\RR^d$ and $f$ be a  non-negative Borel function on $\RR_{+}$. For all $z , y \in \mbox{supp}(\mu)$, we have:
$$ \int f(|x-z|) d\mu(x) = \int f(|x-y|)d\mu(x).$$
\end{theorem}
\begin{proof}
This is a simple application of Fubini's theorem. Indeed, if $f= \alpha \chi_{I}$, where $\alpha\geq 0$ and $I=(c,d)$ is an interval
\begin{align*}
\int f(|x-z|) d\mu(x) &= \alpha \int_{0}^{1} \mu(\left\lbrace x; \chi_{I}(|x-z|) \geq t \right\rbrace) dt, 
                            \\&=\alpha \left( \mu(B(z,d) \cap {B(z,d)}^{c})\right) , 
                            \\&=\alpha \left( \mu(B(y,d)\cap {B(y,c)}^{c}) \right) , \mbox{ since } \mu(B(z,r)) = \mu(B(y,r)) \mbox{ for all }r
                            \\&=\int f(|x-y|) d\mu(x).
\end{align*}
The result follows for general non-negative Borel functions by linearity of the integral and density of step functions.
\end{proof}

Next, we introduce the following beta numbers initially introduced by P. Jones. They quantify how ''flat'' (or far from a plane) the support of a measure is.
\begin{definition}

\begin{itemize}  Let $\mu$ be a Radon measure in $\mathbb{R}^d$, and $\Sigma$ its support. 
\item We define Jones' $\beta^{n}_{\mu}$ number of $B$ to be: $$\beta^{n}_{\mu}(B)=\inf_{L} \sup_{x \in \Sigma \cap B} \frac{\mbox{dist}(x,L)}{r},$$
where $B$ is a ball in $\mathbb{R}^d$, and the infimum is taken over all $n$-planes.
\item We define the bilateral beta number $b\beta^{n}_{\mu}$ of $B$ to be: $$b\beta^{n}_{\mu}(B)= \inf_{L}\left( \sup_{x \in \Sigma \cap B} \frac{\mbox{dist}(x,L)}{r} + \sup_{p \in L \cap B} \frac{\mbox{dist}(p,\Sigma)}{r}\right),$$
where the infimum is taken over all $n$ planes in $\mathbb{R}^d$. We will drop the $n$ superscript and $\mu$ subscript when there is no ambiguity.
\item We say $\mu$ is $n$-flat if there exists an $n$-dimensional plane $V$ in $\mathbb{R}^d$ such that $\mu=\mathcal{H}^{n} \res V$.
\end{itemize}  \end{definition}

Let us define doubling measures.
\begin{definition}
Let $\mu$ be a measure in $\RR^{d}$. We say $\mu$ is a doubling measure if there exists $C>0$ such that:
\begin{equation}
\mu(B(x,2r)) \leq C \mu(B(x,r)), \mbox{ for all } x \in supp(\mu), \mbox{ for all } r>0. 
\end{equation}
The smallest such $C$ is called the doubling constant of $\mu$.
\end{definition}
The two following lemmas relate the weak convergence of a sequence of doubling measures to the convergence of their supports as sets in $\RR^{d}$. They are analogues of Lemmas $[2.2]$ and $[2.3]$ from $\cite{T}$.
\begin{lemma}\label{weakconv1}Let $\mu_{j}$, $\mu$ be doubling Radon measures, all having their doubling constants bounded by the same positive $C>0$. Let $\Sigma_j$, $\Sigma$ be the supports of $\mu_j$ and $\mu$ respectively, and $\overline{B}$ a closed ball in $\RR^d$ such that $\overline{B} \cap \Sigma \neq \emptyset$, and $\overline{B} \cap \Sigma_j \neq \emptyset$ for all $j$. If $\mu_{j}$ converges weakly to $\mu$  ($\mu_j \rightharpoonup \mu$), then $d_{\overline{B},2\overline{B}}(\Sigma_j ,\Sigma )$ converges to $0$, where $d_{\overline{B},2\overline{B}}(U,V)=\sup_{x \in U \cap \overline{B}} dist(x, V \cap 2\overline{B})+ \sup_{x \in V \cap \overline{B}} dist(x, U \cap 2\overline{B})$ . \end{lemma}
\begin{proof}
We first prove that $\sup_{p \in \Sigma_j \cap \overline{B}} dist(p,\Sigma \cap 2\overline{B}) \rightarrow 0$.
Suppose not. 
Then, without loss of generality there exists $\epsilon >0$, $p_j \in \Sigma_j \cap B$, for $j>0$,  such that: $$B(p_j, 2\epsilon) \cap \Sigma \cap 2\overline{B} = \emptyset.$$
In particular, $\mu(B(p_j,2\epsilon))=0$.
Let $\chi_j$, $\tilde{\chi}$ be functions compactly supported in $4\overline{B}$ such that: $\chi_{B(p_j,\epsilon)} \leq \chi_j \leq \chi_{B(p_j, 2 \epsilon)}$, and $\chi_{\overline{B}} \leq \tilde{\chi} \leq\chi_{3\overline{B}}$. 
There exists $k_j \geq 0$ such that 
$$2\overline{B} \subset B(p_j, 2^{k_j} \epsilon).$$
In particular, since $p_j \in \overline{B}$ is closed,
$$k_j \leq \frac{log(2r(B))-log(\epsilon)}{log(2)} \leq K, $$
where $K$ does not depend on $j$.
Since $\mu_j$ are all doubling, we have:
$$C^{-K} \mu_j(2\overline{B}) \leq \mu_j(B(p_j,\epsilon)) \leq \int \chi_j d\mu_j.$$ 
On one hand, $\mu_j(2\overline{B}) \geq \int \tilde{\chi} d\mu_j$ and $ \int \tilde{\chi} d\mu_j \rightarrow \int \tilde{\chi} d\mu >0$ imply that $\liminf \int {\chi_{j}}d\mu_j > 0$.
On the other hand, since $\int \chi_j d\mu_j \leq \int \tilde{\chi} d(\mu-\mu_j) + \mu(B(p_j,2\epsilon))$, then $\int \chi_j d\mu_j  \rightarrow 0$ as $j \rightarrow \infty$, yielding a contradiction.

We now prove that $\sup_{p \in \Sigma \cap \overline{B}} \mbox{dist}(p,\Sigma_j \cap 2\overline{B}) \rightarrow 0$.
Suppose not. Then there exists $\epsilon > 0$, and, without loss of generality, points $x_j \in \Sigma \cap \overline{B}$ such that: $B(x_j, 2\epsilon) \cap \Sigma_j \cap 2\overline{B} = \emptyset$. In particular, $\mu_{j}(B(x_j, 2\epsilon))=0$.
Passing to a subsequence, we can assume that $x_j \rightarrow x$, $x \in \Sigma \cap \overline{B}$ since $\Sigma \cap \overline{B}$.  So there exists $N$ such that, when $j > N$, $|x-x_j| < \epsilon$. Consequently, $B(x, \epsilon) \subset B(x_j,2 \epsilon)$ and $\mu_j(B(x,\epsilon))=0$.
Let $\phi$ be a function compactly supported in $4\overline{B}$ such that $ \chi_{B(x,\frac{\epsilon}{5})} \leq \phi \leq \chi_{B(x,\epsilon)}$.
Then, on one hand, we have: $\int \phi d\mu_j =0$, implying that $\int \phi d\mu= \lim \int \phi d\mu_j=0$. On the other hand, $\int \phi d\mu \geq \mu(B(x, \frac{\epsilon}{5})) > 0$, yielding a contradiction. \end{proof}
\begin{theorem} \label{weakconv2}
Let $\mu_j$, $\mu$ be doubling measures, with the same doubling constant $c$, $B$ a ball such that: $\mu_j \rightharpoonup \mu$, $\Sigma_j \cap B \neq \emptyset$, $\Sigma \cap B \neq \emptyset$. Let $0<n \leq d$.
Then: \begin{equation}\label{eq-weakconv-1}
\frac{1}{2} \limsup \beta^{n}_{\mu_j}(\frac{1}{2}B) \leq \beta^{n}_{\mu}(B) \leq 2 \liminf \beta^{n}_{\mu_j}(2B). \end{equation}
\begin{equation}\label{eq-weakconv-2}\frac{1}{2} \limsup b\beta^{n}_{\mu_j}(\frac{1}{2}B) \leq b\beta^{n}_{\mu}(B) \leq 2 \liminf b\beta^{n}_{\mu_j}(2B). 
\end{equation}
\end{theorem}
\begin{proof} The proof is an easy consequence of Lemma $\ref{weakconv1}$. We prove that: $\beta^{n}_{\mu}(B) \leq 2 \liminf \beta^{n}_{\mu_j}(2B)$ as an example.
Take any $x \in \Sigma_j \cap B$. Let $y \in \Sigma \cap 2B$ be such that $|x-y|=\mbox{dist}(x, \Sigma \cap 2B)$. Pick any $n$-plane $L$ . Then:
$dist(x,L)  \leq \mbox{dist}(x,\Sigma \cap 2B) + \mbox{dist}(y,L)$, implying that $\inf_{L} \sup_{x \in \Sigma_j \cap B} \mbox{dist}(x,L) \leq \sup_{x \in \Sigma_j \cap B} \mbox{dist}(x, \Sigma \cap 2B) + \inf_{L}\sup_{y \in \Sigma \cap 2B} \mbox{dist}(y,L)$.  Therefore, $\beta^{n}_{\mu}(B) \leq 2 \beta^{n}_{\mu_j}(2B)$.

  \end{proof}
  
  To describe the local geometry of a measure, we study objects called its tangents and pseudo-tangents.

\begin{definition} Let $\mu$ be a Radon measure on $\mathbb{R}^d$. 
\begin{itemize} \item We say that $\nu$ is a tangent measure of $\mu$ at 
a point $a \in \mathbb{R}^d$ if $\nu$ is a non-zero Radon measure on $\mathbb{R}^d$, and if there exist sequences $(r_i)$ 
and $(c_i)$ of positive numbers such that $r_i \to 0$ and $c_i T_{a,r_i}\sharp \mu \rightharpoonup \nu$, as $i \to \infty$.
Here, $\mu_i \rightharpoonup \nu$ is a notation for $\mu_i$ converges weakly to $\nu$ and $T_{a,r_i}\sharp \mu$ is 
the push-forward of $\mu$ under the bijection $T_{a,r}(x)=\frac{x-a}{r}$.

\item Let $\Sigma$ denote the support of the measure $\mu$. We say that $\mu$ is $n$-uniform if there exists $c>0$ such that
for all $x \in \Sigma$, for all $r>0$, the following holds: 
$$
\mu(B(x,r))=c r^n. $$  \end{itemize} \end{definition}

In [$\cite{P}$, Theorem $3.11$], Preiss showed that if $\mu$ is an $n$-uniform measure, there exists a unique $n$-uniform measure $\lambda$ such that:
\begin{equation}
r^{-n} T_{x,r} \sharp \mu \rightharpoonup \lambda, \mbox{ as } r \to \infty, 
\end{equation}
for all $x \in \RR^{d}$. $\lambda$ is called the tangent measure of $\mu$ at $\infty$.

A remarkable fact about this measure $\lambda$ is the following ``connectedness at $\infty$'' for the cone of uniform measures. The following is a version of this result formulated by X. Tolsa in $\cite{T}$.

\begin{theorem}[\cite{P}]\label{flat}
Suppose $\mu$ is an $n$-uniform measure in $\RR^d$, $\lambda$ its tangent at $\infty$. 
\begin{itemize}
\item If $n=1,2$, then $\mu$ is flat.
\item If $n \geq 3$, there exists a constant $\tau_0$ depending only on $n$ and $d$ such that, if $\lambda$ satisfies the following: \begin{equation} \beta_{\lambda}^{n}(B(0,1)) \leq \tau_0, \end{equation} then $\mu$ is $n$-flat. 
\end{itemize} \end{theorem}

Another notion of interest is that of pseudo-tangent measures introduced by Toro and Kenig in $\cite{KT}$.
\begin{definition} 
Let $\mu$ be a doubling Radon measure in $\RR^{d}$. We say that $\nu$ is a pseudo-tangent measure of $\mu$ at the point $x \in \mbox{supp} \mu$ if $\nu$ is a nonzero Radon measure in $\RR^{d}$ and if there exists a sequence of points $x_i \in \mbox{supp} \mu$ such that $x_i \to x$ and a sequence of positive numbers $\left\lbrace r_i \right\rbrace$ such that $r_i \downarrow 0$ and ${r_i}^{-n} T_{{x_i},{r_i}} \sharp \mu \rightharpoonup \nu$.
\end{definition}

Let us define the notion of asymptotically optimally doubling measures.
\begin{definition} If $x \in \Sigma$, $r>0$ and $t \in (0,1]$, define the quantity:
\begin{equation}
R_{t}(x,r)= \frac{\mu(B(x,tr))}{\mu(B(x,r))} - t^{n}.
\end{equation}
We say $\mu$ is asymptotically optimally doubling if for each compact set $K \subset \Sigma$, $x \in K$, and $t \in [\frac{1}{2},1]$ 
\begin{equation}\label{asymptoptim}
\lim_{r \to 0^{+}} \sup_{x \in K} \left| R_{t}(x,r) \right| = 0.
\end{equation}

\end{definition}
The following theorem is a useful feature of pseudo-tangent measures: they turn out to be $n$-uniform if the measure they originate from is asymptotically optimally doubling.

\begin{theorem}[\cite{KT}]\label{pseudounif} Let $\mu$ be a Radon measure in $\mathbb{R}^{d}$ that is doubling and $n$-asymptotically optimally doubling. Then all pseudo-tangent measures of $\mu$ are $n$-uniform. 
\end{theorem}

We define Ahlfors regular and locally Ahlfors regular measures. 

\begin{definition} Let $\mu$ be a Radon measure in $\mathbb{R}^d$, and $\Sigma$ its support.
\begin{itemize} \item We say $\mu$ is Ahlfors $n$-regular, $0 <n \leq d$ if there exists a constant $c_1$ such that:
 \begin{equation}\label{regular} c_1^{-1} r^n \leq \mu(B(x,r)) \leq c_1 r^n, \mbox{  for all } x \in \Sigma, \; r>0. 
 \end{equation}
 \item  We say $\mu$ is locally Ahlfors $n$-regular if for all $K$ compact, there exist constants $c_K>0$ and $r_K$  such that, for all $x \in \Sigma \cap K$, $0<r \leq r_K$, 
\begin{equation*} c_K^{-1} r^n \leq \mu(B(x,r)) \leq c_K r^n .
\end{equation*} \end{itemize} \end{definition}

A usefool tool to obtain discrete versions of the Jones beta numbers is to decompose $\Sigma$ into dyadic cubes. David proved that such a dyadic decomposition into $\mu$-cubes exists for the support of Ahlfors-regular measures $\mu$ in $\cite{D}$. Christ generalized this decomposition to spaces of homogeneous type in $\cite{C}$.
\begin{theorem}[Dyadic Decomposition 1, \cite{D}] Given an Ahlfors-regular measure $\mu$, $\Sigma$ its support, the following holds. 
For each $j \in \mathbb{Z}$, there exists a family $\mathcal{D}_j$ of Borel subsets of $\Sigma$ (the dyadic cubes of the $j$-th generation) such that:
\begin{itemize}
\item each $\mathcal{D}_j$ is a partition of $\Sigma$.
\item if $Q \in \mathcal{D}_j$,  $Q' \in \mathcal{D}_k$ with $k \leq  j$ , then either $Q \subset Q'$ or $Q \cap Q'= \emptyset$.
\item for all $j \in \mathbb{Z}$ and $Q \in \mathcal{D}_j$ , we have $diam(Q) \sim 2^{-j}$ and $c^{-1} 2^{-jn} \leq \mu(Q) \leq c 2^{-jn}$.
\item if $Q \in \mathcal{D}_j$, there exists some point $z_Q \in Q$ (the center of $Q$) such that
$dist(z_Q,  \Sigma \backslash Q) \geq c 2^{-j}$. \end{itemize}  \end{theorem}
\begin{definition}
A space of homogeneous type is a set $X$, equipped with:
\begin{itemize} \item a quasi metric $d$ for which all the associated balls are open. The constant from the weakened triangle inequality is denoted $A_0$.
\item a nonnegative, Borel, locally finite measure $\mu$ satisfying the doubling condition: \begin{equation*} \mu(B(x,2r)) \leq A_1 \mu(B(x,r)). \end{equation*} \end{itemize} \end{definition}
\begin{theorem}[Dyadic Decomposition 2, \cite{C}] \label{dyadic}
Suppose $X$ is a space of homogeneous type. Then, there exist constants $\delta \in (0,1)$, $a_0>0$, $\gamma>0$, and, for each $j \in \mathbb{Z}$, there exists a countable collection of open subsets $\mathcal{D}_{j}^{\mu}$ (the dyadic cubes of generation $j$) such that:
\begin{enumerate}
\item $\mu(X \backslash \bigcup_{Q \in \mathcal{D}_j} Q)=0.$
\item If $k \geq j$, $Q \in \mathcal{D}_k$, $Q' \in \mathcal{D}_j$, then either $Q \subset Q'$ or $Q \cap Q'=\emptyset$.
\item For each cube $Q \in \mathcal{D}_j$ and $k<j$, there exists a unique $Q' \in D_k$ such that : $Q \subset Q'$.
\item If $Q \in \mathcal{D}_j$, then $\rm{diam}(Q) \lesssim \delta^j$.
\item Each $Q \in \mathcal{D}_j$ contains some ball $B(z_Q,a_0\delta^j)$.
\item If $Q \in \mathcal{D}_j$, $\mu(\{ x \in Q; d(x, X\backslash Q)\leq t\delta^j\}) \lesssim t^{\gamma} \mu(Q)$, for all $t>0$.
\end{enumerate}
\end{theorem}
We denote $\mathcal{D}^{\mu} =\cup_{j \in \mathbb{Z}} mathcal{D}^{\mu}_j$ . Given $Q \in mathcal{D}^{\mu}_j$ , the unique cube $Q' \in mathcal{D}^{\mu}_{j-1}$ which
contains $Q$ is called the parent of $Q$. We say that $Q$ is a child of $Q'$. Also, given
$Q \in mathcal{D}^{\mu}$, we denote by $mathcal{D}^{\mu}(Q)$ the family of cubes $P \in mathcal{D}^{\mu}$ which are contained in
$Q$. We also denote by $mathcal{D}^{\mu}_{j}(Q)$ the descendants of $Q$ of generation $j$.
For $Q \in mathcal{D}^{\mu}_j$, we define the side length of $Q$ as $l(Q) = 2^{-j}$ . Notice that $c l(Q) \leq \mbox{diam}(Q) \leq l(Q)$.

For each $Q \in mathcal{D}^{\mu}$, we define $B_Q$ to be the ball $B(z_Q, 3l(Q))$, and the corresponding coefficients $\beta^{n}_{\mu}(Q)= \beta_{\mu}^n(B_Q)$ and $b\beta^{n}_{\mu}(Q)= b\beta_{\mu}^n(B_Q)$. 

In Theorem $\ref{dyadic}$, $\delta$ depends only on $A_0$. In fact, if $d$ is a metric (in which case $A_0=1$), $\delta$ can be replaced by $\frac{1}{2}$.

Putting (4) and (5) from Theorem $\ref{dyadic}$ together, we get: $\mathrm{diam}(Q) \sim 2^{-j}$ if $Q \in \mathcal{D}_j$.

If $\mu$ is locally Ahlfors $n$-regular, (4) and (5) also imply that if we fix a ball $B(0,\rho)$, then for the cubes $Q$ of $\mathcal{D}_j$ intersecting $B(0,\rho)$, such that $\mathrm{diam}(Q) \leq \rho$, the following holds : \begin{equation} c_{\rho}^{-1} 2^{-jn} \leq \mu(Q) \leq c_{\rho} 2^{-jn}. \end{equation}
 When $d$ is a metric, we will call the point $z_Q$ from (5) in Theorem $\ref{dyadic}$ the center of $Q$. Note that: $\mathrm{dist}(z_Q, X \backslash Q) \gtrsim 2^{-k}$.
 
Using Theorem $\ref{dyadic}$ and the remarks above, we can generalize David's dyadic decomposition to locally Ahlfors regular measures
 \begin{corollary} \label{Dyadic3}
Let $\mu$ be a  doubling and locally Ahlfors $n$-regular measure in $\RR^{d}$, and $D^{\mu}$ the dyadic decomposition of its support from Theorem $\ref{dyadic}$.
Then:\begin{enumerate}
\item each $\mathcal{D}_j$ is a partition of $\Sigma$.
\item if $Q \in \mathcal{D}_j$,  $Q' \in \mathcal{D}_k$ with $k \leq  j$ , then either $Q \subset Q'$ or $Q \cap Q'= \emptyset$.
\item for all $j \in \mathbb{Z}$ and $Q \in \mathcal{D}_j$ and $Q \cap K \neq \emptyset$ where $K$ is a compact set , we have $diam(Q) \sim 2^{-j}$ and $c_{K}^{-1} 2^{-jn} \leq \mu(Q) \leq c_{K} 2^{-jn}$.
\item if $Q \in \mathcal{D}_j $, there exists some point $z_Q \in Q$ (the center of $Q$) such that
$dist(z_Q,  \Sigma \backslash Q) \geq c 2^{-j}$.
\end{enumerate}
\end{corollary}
\begin{definition}Let $\mu$ be a doubling measure in $\RR^{d}$. We say that $\mathcal{F} \subset D^{\mu}$ is a Carleson family if there exists some constant $c>0$ such that:
$$ \sum_{Q \in \mathcal{F}, Q \subset R} \mu(Q) \leq c \mu(R),\mbox{   for all } R \in D^{\mu}.$$ \end{definition}

The notion of uniform rectifiability was introduced by David and Semmes in $\cite{DS2}$. 
It is a quantitative version of the notion of $n$-rectifiability. 
\begin{definition} Let $\mu$ be a Radon measure in $\mathbb{R}^d$, and $\Sigma$ its support.
 \begin{itemize} \item We say $\mu$ is uniformly $n$-rectifiable if it is Ahlfors $n$-regular, and there exist constants $\theta$ and $M$ so that,
 for each $x \in \Sigma$ and $R>0$, there is a Lipschitz mapping $g$ from $ \mathbb{R}^n$ to $\mathbb{R}^d$ such that $g$ has Lipschitz norm not exceeding $M$ and such that:
$$\mu(B(x,R) \cap g(\mathbb{R}^n)) \geq \theta R^n.$$
We say $\Sigma$ has big pieces of Lipschitz images (BPLI). 

\item We say $\mu$ is $\bf{locally}$ uniformly $n$-rectifiable  if it is locally Ahlfors $n$-regular, and for every compact set $K$, there exist constants $R_{K}$, $\theta_{K}$ and $M_{K}$ so that,
 for each $x \in \Sigma \cap K$ and $0<R \leq R_{K}$, there is a Lipschitz mapping $g$ from $\mathbb{R}^n$ to $\mathbb{R}^d$ such that $g$ has Lipschitz norm not exceeding $M_{K}$ and such that:
$$\mu(B(x,R) \cap g(\mathbb{R}^n) )\geq \theta_{K} R^n.$$
\end{itemize}
\end{definition}

We now define some notions related to uniform rectifiability. David and Semmes proved the following theorem in $\cite{DS2}$:
\begin{theorem}[Bilateral Weak Geometric Lemma,\cite{DS2}]\label{BWGL}
Let $\mu$ be an Ahlfors-regular measure in $\RR^{d}$. Then $\mu$ is uniformly rectifiable if and only if it satisfies the Bilateral Weak Geometric Lemma namely:
for all $\eta >0$, the family $B(\eta)= \{ Q \in D^{\mu}; b\beta^{n}_{\mu}(Q) > \eta\}$ is a Carleson family. \end{theorem}

\begin{definition} Let $\mu$ be a doubling, locally Ahlfors $n$-regular measure. For $\rho >0$, we define the local dyadic decomposition $\mathcal{D}_{\rho}$ as :
\begin{equation}\nonumber \mathcal{D}_{\rho}= \{ Q \in D^{\mu}, Q \cap B(0,\rho) \neq \emptyset, \mathrm{diam}(Q)\leq \rho \}. \end{equation}
We say $\mu$ satisfies the Bilateral weak geometric lemma locally if for all $\rho>0$ and for all $\eta >0$, the family $B_{\rho}(\eta)= \{ Q \in D_{\rho}; \;  b\beta^{n}_{\mu}(Q) > \eta\}$ is a Carleson family. \end{definition}
Since the arguments in the proof of Theorem $\ref{BWGL}$ are local (in the sense that the dyadic decomposition is divided into maximal dyadic cubes, and the result is proven on each of these cubes) and since Corollary $\ref{Dyadic3}$  implies that locally the setting is the same as David's dyadic decomposition, the results of Theorem $\ref{BWGL}$ hold locally.
\begin{theorem}[Local Bilateral Weak Geometric Lemma]
Let $\mu$ be a doubling, locally Ahlfors $n$-regular measure in $\RR^{d}$. Then $\mu$ is locally uniformly rectifiable if and only if it satisfies the Bilateral Weak Geometric Lemma locally i.e., at scale $\rho$  for some $\rho>0$ namely:
for all $\eta >0$  the family $B_{\rho}(\eta)= \{ Q \in D_{\rho}; \;  b\beta^{n}_{\mu}(Q) > \eta\}$ is a Carleson family. \end{theorem}

Our final definition is of the Riesz transforms of a measure.
\begin{definition}
Let $\mu$ be a Radon measure in $\RR^{d}$. The Riesz transform of $\mu$ for $z_0 \in supp(\mu)$, $0 < r < s$ is defined as :
 $$R_{r,s} \mu(z_0)= \int_{r\leq |z_0-y| \leq s} \quad \frac{z_0-y}{|z_0-y|^{n+1}} d\mu(y). $$
\end{definition}
In $\cite{T}$, the following estimate on the Riesz transform was an essential tool for Tolsa's proof of the uniform rectifiability of $n$-uniform measures.
\begin{lemma}[\cite{T}]\label{bigRiesz}
  Let $\mu$ be a Radon measure, $\Sigma = \mbox{supp} (\mu)$, $n\leq d$. Let $B$ be a ball centered in $\Sigma$.
 Suppose that there exist constants $\kappa$, $c_1$ such that:
 \begin{equation}\label{kindareg}c_1^{-1} r^n \leq \mu(B(x,r)) \leq c_1 r^n, \mbox{ for} \; x \in B \cap \Sigma, \kappa r(B) \leq r \leq r(B).
 \end{equation} 
 Moreover, suppose that for some $\epsilon>0$, we have: 
 \begin{equation}\label{notflat}\beta^{d-1}_{\mu}(B(x,r)) \geq \epsilon, \mbox{ for}\;  x \in B \cap \Sigma, \kappa r(B) \leq r \leq r(B). 
 \end{equation}
 Then, for any $M>0$, there exists $\kappa_0$  ($\kappa_0= \kappa(M,\epsilon,c_1)$), such that if $\kappa \leq \kappa_0$ , then there exists $r$, $\kappa r(B) \leq r \leq r(B)$, and points $x,z_0 \in B \cap \Sigma$, with $|x-z_0| \leq \kappa r(B)$ satisfying:
 \begin{equation} \label{bounded} \left| \frac{x-z_0}{\kappa r(B)} . R_{\kappa r(B),r}\mu(z_0) \right| \geq M. \end{equation}
 \end{lemma}

\section{Existence of big flat balls for uniformly distributed measures}
We start by proving that the Riesz transform of a uniformly distributed measure is locally bounded.
The two following lemmas are local analogues to Lemmas 3.1 and 3.4 in $\cite{T}$ for uniformly distributed measures.
 \begin{lemma}\label{smallRiesz} Let $\mu$ be a uniformly distributed measure, $dim_0(\mu)=n$, $R>0$. Let $z_0 \in \Sigma$, $0<r \leq R$. Then we have: $$ \left| \frac{x-z_0}{r}\,.\,R_{r,s} \mu(z_0) \right| \leq c, \mbox{ for all }r <s \leq \frac{R}{2}, \mbox{ and for all } x \in B(z_0,r)$$ where $c$ depends only on $R$. \end{lemma}   
 \begin{proof}
 Without loss of generality, assume $z_0=0$.
 For $r$,$s$ fixed, $0<r<s\leq \frac{R}{2}$, define the function $\psi : \mathbb{R} \rightarrow \mathbb{R}$ to be a compactly supported $C^{\infty}$ function with the following properties:
 $$\psi(t) = \begin{cases} 0 & \mbox{if} \;  |t| \geq 2s \; \mbox{or} \; |t| \leq \frac{r}{2} \\ \frac{1}{t^{n}} & \mbox{if}  \; r \leq |t| \leq s \end{cases}.$$
 and 
 \begin{equation}\label{phi}|\psi(t)| \leq c \; \mbox{min}\left( \frac{1}{r^{n}}, \frac{1}{t^{n}}\right) \; \mbox{for all} \; t \in \mathbb{R} . 
 \end{equation}
 We also require that: 
 \begin{equation} \label{phi'} 
 | \psi'(t) | \leq c \; \mbox{min} \left( \frac{1}{(3r)^{n+1}}, \frac{1}{t^{n+1}} \right) , \mbox{for all }  t \in \mathbb{R}.
 \end{equation}
 We define real-valued functions $\rho$ and $\Psi$ respectively from $\mathbb{R}$ and $\mathbb{R}^d$ as follows:
 $$\rho(u) = -\int_{u}^{\infty} \: \psi(t) dt, \; u \in \mathbb{R},$$
 and $$\Psi(y)= \rho(|y|), \; y \in \mathbb{R}^d.$$
 
 Since $\mu$ is uniformly distributed and $\Psi$ is radial, for all $x \in supp(\mu)$, we have by Theorem $\ref{radial}$
 \begin{equation}\label{unifeq}\int \: \Psi(x-y) d\mu(y) - \int \: \Psi(y) d\mu(y) = 0. \end{equation}
 
 On the other hand, Taylor's formula gives:
 \begin{equation}\label{taylor1}\Psi(x-y) - \Psi(-y)= x \cdot \nabla \Psi(-y)+ \frac{1}{2} x^{T} \cdot \nabla^2\Psi(\xi_{x,y}) \cdot x, \; \mbox{where} \; \xi_{x,y} \in [x-y,-y] \subset \mathbb{R}^d .
 \end{equation} 
 Note that: $\nabla \Psi(z)=\psi(|z|) \cdot \frac{z}{|z|}$ implying that:
 \begin{equation}\label{taylor2} - R_{r,s}\mu(0)= \int_{r \leq |y| \leq s} \nabla \Psi(-y) d\mu(y) \end{equation}
 since $\psi(|z|)=\frac{1}{|z|^n}$ when $|z| \in (r,s)$.
Combining ($\ref{unifeq}$),($\ref{taylor1}$) and ($\ref{taylor2}$), we get
$$ x \cdot \int_{|y| \leq r} \nabla \Psi(-y) d\mu(y) - x  \cdot R_{r,s}\mu(0) + x \cdot \int_{|y| > s} \nabla \Psi(-y) d\mu(y)+ \frac{1}{2} x^{T}. \left(\int \nabla^2\Psi(\xi_{x,y})d\mu(y)\right) \cdot x \:=0.$$
This gives: \begin{equation}\label{taylor3} x.R_{r,s} \mu(0) = x\;.\; \int_{\left\lbrace |y| \leq r \right\rbrace \cup \left\lbrace |y| > s \right\rbrace} \nabla \Psi(-y) d\mu(y) + \frac{1}{2} x^{T}.\left(\int \nabla^2\Psi(\xi_{x,y}) d\mu(y)\right).\; x .
\end{equation}

Let us estimate the right hand-side of ($\ref{taylor3}$).
\\Using the inequalities ($\ref{phi}$) and ($\ref{phi'}$) and Corollary $\ref{locallyAR}$, we get the following estimates on the first term in ($\ref{taylor3}$):
$$\left| \int_{|y| \leq r} \nabla \Psi(-y) d\mu(y) \right| \leq \frac{c}{r^{n}} \mu(B(0,r)) \leq C \; \mbox{ since }r \leq R$$
and since $2s \leq R$
$$\left| \int_{|y|>s} \nabla \Psi(-y) d\mu(y) \right| = \left| \int_{s<|y|\leq 2s} \nabla \Psi(-y) d\mu(y) \right| \leq\frac{1}{s^{n}} \mu(B(0,2s)) \leq C\: 2^{n}, $$
where $C$ depends on $R$.
Let us now estimate the second order derivative of $\Psi$ using the fact that: $$\left| \nabla^2 \Psi(\xi_{x,y}) \right| \leq c \; \mbox{min} \left(\frac{1}{3r^{n+1}}, \frac{1}{|\xi_{x,y}|^{n+1}}\right).$$

If $|y| \leq 2r$, we get: $\left| \nabla^2 \Psi(\xi_{x,y}) \right| \leq \frac{C} {r^{n+1}}$.  

If $|y| > 2r$, then $\frac{|y|}{2} \leq |x-y| \leq 2|y|$ implies that $ |\xi_{x,y}| \sim |y|$ and hence: $\left| \nabla^2 \Psi(\xi_{x,y}) \right| \leq \frac{C}{|y|^{n+1}}$.

Therefore, since $\mu$ is uniformly distributed, and $\psi$ compactly supported in $[-R,R]$, we get:
\begin{align*}
\int \left| \nabla^2 \Psi(\xi_{x,y}) \right| d\mu(y) & \leq c \int_{|y| \leq 2r} \frac{1}{r^{n+1}} d\mu(y) + c \int_{2r< |y| \leq R} \frac{1}{|y|^{n+1}} d\mu(y), \\ & \leq \frac{c}{r} + c \int_{2r< |y| \leq R} \frac{1}{|y|^{n+1}} d\mu(y), \mbox{ since } 2r \leq R .
\end{align*} 
We claim that:
$$ \int_{2r<|y| \leq R} \frac{1}{|y|^{n+1}} d\mu(y) \leq \frac{c}{r}. $$
Indeed
\begin{align*}
\int_{2r < |y| \leq R} \frac{1}{|y|^{n+1}} d\mu(y) & = \int_{\frac{1}{R^{n+1}}}^{\frac{1}{(2r)^{n+1}}} \mu \left( \left\lbrace y : |y| < \frac{1}{t^{\frac{1}{(n+1)}}} \right\rbrace \right) dt, \\
&=  \int_{\frac{1}{R^{n+1}}}^{\frac{1}{(2r)^{n+1}}} \mu \left(B\left(0,\frac{1}{t^{(n+1)}}\right)\right) dt, \\
& \leq c  \int_{\frac{1}{R^{n+1}}}^{\frac{1}{(2r)^{n+1}}} \frac{1}{t^{\frac{n}{(n+1)}}}dt, \\
&= c \left(\frac{1}{2r} - \frac{1}{R} \right) \leq \frac{c}{2r}.
\end{align*}

This gives:
$$\left|x \;.\;R_{r,s} \mu(0) \right| \leq c |x|+ c \frac{|x|^2}{r} \leq cr $$

since $|x| \leq r$.

 \end{proof}    
 \begin{lemma}\label{inductflat} Let $\mu$ be a uniformly distributed measure, $n=dim_{0}\mu$, $n<m \leq d$, $\epsilon >0$. Let $K$ be a compact set, $R=diam(K)$.
 There exist constants $\delta$, $\tau$ depending only on $\epsilon$, $K$, $n$ and $d$ such that: if B is a ball centered in $\Sigma \cap K$, with $r(B) \leq R$ and $\beta_{\mu}^{m}(B) \leq \delta$, then there exists a ball $B'$, $B' \subset B$, centered in $\Sigma$ such that: $\beta_{\mu}^{m-1}(B') \leq \epsilon$ and $r(B') \geq \tau \:.\: r(B)$.
 \end{lemma}
 \begin{proof} 
 Let $L$ be a best approximating $m$-plane for $\beta_{\mu}^{m}(B)$. In particular, for any $z \in \Sigma \cap B$ , 
 \begin{equation}\label{flatdelta} |z-\pi(z)|<\delta r(B) ,\end{equation}
 assuming that $\beta_{\mu}^{m}(B) \leq \delta$ for a $\delta$ to be chosen later.
 Denote $\pi_L$, the orthogonal projection onto $L$,  by $\pi$ and define the measure $\nu$ on $L$ to be:
 $$ \nu(A)=\mu(\pi^{-1}(A) \cap B) \quad \mbox{ for } A \subset L .$$
 We also denote the radius of $B$ by $r(B)$.
 Let $x \in \frac{1}{2} B \cap \widetilde{\Sigma}$, where $\widetilde{\Sigma}= \mbox{supp} \, \nu$.
On one hand, denoting by $B_L(x,r)$ the ball in $L$ i.e. $B(x,r) \cap L$, we have $B(x,r) \subset \pi^{-1}({B_L(x,r)}) \cap B$ when $r \leq \frac{r(B)}{2}$, and hence:
$$  c^{-1} r^{n} \leq \mu(B(x,r)) \leq \nu({B_L(x,r)}).$$
On the other hand, if $J$ is the maximal number of disjoint balls of radius $r$ that can be contained in $   \pi^{-1}({B_L(x,r)}) \cap 2B$,  then we claim that 
$$J \leq \omega_{m} (2r)^{m-d} {2r(B)}^{d-m}.$$ 
Indeed, assuming that $L$ is the $m$-plane 
$\left\lbrace (y_1, \ldots ,y_d); y_{m+1}=\ldots =y_d=0 \right\rbrace $ ,  the union $E$ of these $J$ balls is included in the cylinder  $C=\left\lbrace y; (y_1,\ldots ,y_m) \in B_{L}(x,r), y_j \in (-2r(B),2r(B)), j>m \right\rbrace$.
Thus
\begin{align*}
J. \mathcal{L}^{d}(B(0,r))  & \leq \mathcal{L}^{d}(\mathcal{C}), \\ &= \omega_{m} r^{m}{2r(B)}^{d-m}. \end{align*}
Let $\left\lbrace B_i \right\rbrace$ be a disjoint collection of balls of radius $r$ such that $B_i \subset \pi^{-1}({B_L(x,r)}) \cap 2B  \subset \cup 5B_i $, obtained by a Vitali covering. In particular, by the above argument, there are at most $J$ of them. Then, if $\delta r(B) \leq r \leq r(B)$, where $\delta$ is to be chosen later, we have:
$$ \nu({B_L(x,r)})\leq \mu(\cup 5 B_i) \leq J.c.(r)^{n} \leq  (2r)^{m-d} \omega_{m} r(B)^{d-m}  (2r)^{n} \leq \frac{c}{\delta^{d-m}} r^n. $$

Hence, letting $C=\frac{c}{\delta^{d-m}}$, we have:
$$ C^{-1} r^{n} \leq  \nu(B(x,r)) \leq C r^{n}, \quad \mbox{when} \; x \in \frac{1}{2} B \cap \widetilde{\Sigma},\; \delta r(B) \leq r \leq r(B).$$
We first claim that for all $z_0 \in \frac{1}{2}B \cap \widetilde{\Sigma}$, and $r_0, r$ with $\delta^{\frac{1}{2}}r(B) \leq r_0 \leq r \leq r(B)$, if $\delta$ is small enough, then:
\begin{equation} \label{Rieszestimate}  \left| \frac{x-z_0}{r_0} R_{r_0,r}\nu(z_0) \right| \leq c, \quad \mbox{for} \; x\in \widetilde{\Sigma} \cap B(z_0,r_0).
\end{equation}
Let us finish the proof of the lemma before proving $\eqref{Rieszestimate}$.
Choose $\epsilon>0$. Let $\kappa_0=\kappa_0(\epsilon)$ be as in Lemma $\ref{bigRiesz}$. Let $\delta$ be small enough for $\eqref{Rieszestimate}$ to hold,  $\delta^{\frac{1}{2}} \leq \kappa_0$ and $\delta^{\frac{1}{2}} \leq \frac{\epsilon}{2}$.
 Identify $L$ with $\mathbb{R}^{m}$. Since $\nu$ satisfies $\eqref{kindareg}$ but not  $\eqref{bounded}$, then $\eqref{notflat}$ cannot hold. Namely, there must exist a ball $B'$ centered in $\widetilde{\Sigma}$, and $r(B') \geq \delta^{\frac{1}{2}} r(B)$, such that : $\beta_{\nu}^{m-1}(B') \leq\frac{1}{2}\epsilon$. We also have  $B' \subset \frac{1}{2}B$ by the same argument as in the proof of Theorem $\ref{bigRiesz}$ in $\cite{T}$.
Let $L'$ be a best approximating $(m-1)$-plane for $\beta_{\nu}^{m-1}(B')$.
We claim that the following holds:
 \begin{equation} \label{nhoodL'} \Sigma \cap B' \subset U_{\delta r(B)+\frac{\epsilon r(B')}{2}}(L') ,
 \end{equation} 
 where $ U_{\delta r(B)+\frac{\epsilon r(B')}{2}}(L')$ denotes the $(\delta r(B)+\frac{\epsilon r(B')}{2})$-neighborhood of $L'$.
Indeed , we have: \begin{equation}\label{betaL'} \beta_{\nu}^{m-1}(B') \leq \frac{\epsilon}{2}. 
\end{equation}
Suppose $z \in \Sigma \cap B'$. Then $z=\pi(z) + \pi_{L^{\bot}}(z)$, where $\pi$ and $\pi_{L^{\bot}}$ are the projections onto $L$ and $L^{\bot}$ respectively.
Since $\Sigma \cap B' \subset \Sigma \cap B$ and $\beta_{\mu}^{m}(B) \leq \delta$, we get: $\left| \pi_{L^{\bot}}(z) \right| \leq \delta r(B)$.
Moreover, by ($\ref{betaL'}$), $\mbox{dist}(\pi(z),L') \leq \frac{\epsilon}{2}r(B')$.
Thus, for all $z \in \Sigma \cap B'$,  \begin{align*}
dist(z,L') &\leq |z-\pi(z)|+ dist(\pi(z),L'), \\
& \leq  \delta r(B) + \frac{\epsilon}{2}r(B').
\end{align*}
Therefore:
$$\beta_{\mu}^{m-1}(B') \leq \frac{\delta r(B)+ \frac{\epsilon}{2}r(B')}{r(B')} \leq \delta^{\frac{1}{2}}+\frac{\epsilon}{2} \leq \epsilon. $$

We now prove $\eqref{Rieszestimate}$. Pick $z_0 \in \frac{1}{2} B \cap \widetilde{\Sigma} , \, \, r_0,r \, \mbox{such that  } \delta^{\frac{1}{2}} r(B) \leq r_0 < r \leq r(B) $. Choose any $x \in \widetilde{\Sigma} \cap B(z_0, r_0)$, and let $z_1, x_1 \in \Sigma \cap B$ be such that: $\pi(z_1)= z_0,  \, \pi(x_1)=x$.
Then:
\begin{equation}\label{ABC} 
 \left| \frac{x-z_0}{r_0} R_{r_0,r}\nu(z_0) \right|  \leq A + B + C . 
\end{equation}

where \begin{equation}\label{Aeq} A = \left| \frac{x-z_0}{r_0} \cdot \left( R_{r_0,r}\nu(z_0)-  R_{r_0,r}\mu(z_1) \right) \right|, \end{equation}
\begin{equation}\label{Beq} B =  \left| \frac{(x-z_0)-(x_1-z_1)}{r_0} \cdot R_{r_0,r}\mu(z_1) \right| \end{equation} and
\begin{equation}\label{Ceq} C=\left| \frac{x_1 - z_1}{r_0} \cdot R_{r_0,r}\mu(z_1) \right| . \end{equation}

Let us first estimate ($\ref{Aeq}$). We denote the kernel of the Riesz transform by $K$, and the annulus in $\mathbb{R}^d$ by $A(z_0,r_0,r)= \{y \in \mathbb{R}^d; r_{0} < |y-z_0| \leq r\}$, the annulus in $L$ by $A_L(z_0,r_0,r)$.
Then, we can write:
$$R_{r_0,r} \nu(z_0) = \int_{B \cap \pi^{-1}(A_{L}(z_0,r_0,r))} K(z_0-\pi(y)) d\mu(y). $$
And: 

\begin{align}\label{S1S2}
\left| R_{r_0,r} \nu (z_0) - R_{r_0,r} \mu(z_1) \right| & = \left| \int_{B \cap \pi^{-1}(A_{L}(z_0,r_0,r))} K(z_0-\pi(y)) d\mu(y) - \int_{A(z_1,r_0,r)} K(z_1-y) d\mu(y)\right| \nonumber \\ & \leq S_1 + S_2 . 
\end{align} 
 
 where $$S_1:=\int_{B \cap \pi^{-1}(A_{L}(z_0,r_0,r))} |K(z_0-\pi(y))-K(z_1-y)| d\mu(y), $$
and
$$S_2:=\left| \int_{B \cap \pi^{-1}(A_{L}(z_0,r_0,r))} K(z_1-y) d\mu(y)-\int_{A(z_1,r_0,r)} K(z_1-y) d\mu(y)\right|. $$
\\Estimating $S_1$ from ($\ref{S1S2}$):

To estimate $S_1$, we need the following  intermediate estimates.
First, note that if $y \in B \cap \pi^{-1}(A(z_0, r_0, r)) \cap \Sigma$, it follows from the fact that $ |z_0-y|^2=|z_0-\pi(y)|^2+|\pi(y)-y|^2$ and $\pi(y) \in A(z_0, r_0, r)$ that the following holds \begin{equation}\label{pyth} |z_0-y| \geq |z_0-\pi(y)| \geq  r_0 \end{equation}

Second, by $\eqref{flatdelta}$, \begin{equation}\label{estimateS1} |(z_0-\pi(y))-(z_1-y)| \leq |z_0-z_1| + |y-\pi(y)| \leq 2 \delta r(B). 
\end{equation}
Similarly, we claim that : \begin{equation}\label{estimateS1'} \frac{1}{2}|z_0-y| \leq |z_0 - \pi(y)| \leq 2 |z_0-y| \; \mbox{and} \; \frac{1}{2} |z_0-y| \leq |z_1-y| \leq 2 |z_0 - y|. \end{equation}
Indeed, on one hand assuming $\delta$ small enough that $\delta r(B) \leq \frac{1}{2} r_0$,  \begin{align*}|z_0-y| & \leq |z_0 - \pi(y)| + |\pi(y)-y|,\\
&\leq |z_0 - \pi(y)| +\delta r(B) , \\
&\leq |z_0 - \pi(y)|+\frac{1}{2} r_0, \\
&\leq 2 |z_0 - \pi(y)|, \mbox{ by } \eqref{pyth} .
\end{align*}
On the other hand: \begin{align*}
 |z_0-\pi(y)| &\leq |z_0-y|+|y-\pi (y)|, \\
 & \leq |z_0-y|+ \frac{1}{2} r_0, \\ & \leq 2 |z_0-y|, \mbox{ by } \eqref{flatdelta}.
\end{align*}
The other estimate in $\eqref{estimateS1'}$ follows similarly. On one hand,
\begin{align*}
|z_0-y| &\leq |z_0 - z_1| + |z_1 - y|, \\ &\leq \delta r(B)+ |z_1-y|, \\ & \leq |z_0 - \pi(y)|+|z_1 - y|, \\ & \leq 2 |z_1-y|, \mbox{ since } z_0=\pi (z_1).
\end{align*}
On the other hand, 
\begin{align*}
|z_1 - y| & \leq |z_1-z_0| + |z_0 - y|, \\
& \leq \delta r(b)+ |z_0 - y|, \\
& \leq 2 |z_0 - y|.
\end{align*}

Thus, to estimate $S_1$, noting that:

\begin{align} \left| K(z_0-\pi(y)) - K(z_1 - y) \right| &\leq \left| \frac{z_0-\pi(y)}{|z_0-\pi(y)|^{n+1}} - \frac{z_1-y}{|z_0-\pi(y)|^{n+1}} \right| + \left| \frac{z_1-y}{|z_0-\pi(y)|^{n+1}} - \frac{z_1-y}{|z_1-y|^{n+1}} \right| \nonumber\\ & =: D_1 + D_2, \end{align}

one obtains on one hand, using ($\ref{pyth}$) and ($\ref{estimateS1}$), \begin{equation}\label{D1}D_1= \frac{|z_0-\pi(y)-(z_1-y)| }{|z_0-\pi(y)|^{n+1}} \leq \frac{2^{n+1} C \delta r(B)}{|z_0-y|^{n+1}}, \end{equation}
and on the other hand, using ($\ref{estimateS1}$) and ($\ref{estimateS1'}$), \begin{equation}\label{D2}\begin{aligned} D_2 &=|z_1-y| \left| \frac{ |z_1-y|^{n+1} - |z_0-\pi(y)|^{n+1}}{|z_1-y|^{n+1} |z_0-\pi(y)|^{n+1}}\right| \\ & = |z_1-y| \frac{\left||z_1-y| - |z_0-\pi(y)|\right|}{|z_1-y|^{n+1} |z_0-\pi(y)|^{n+1}} \cdot \Sigma_{0}^{n}|z_1-y|^{n-j} |z_0-y|^{j} \\ &\leq C_n \frac{|z_0-y|^{n+1}}{|z_0-y|^{2n+2}}. \left||z_1-y| - |z_0-\pi(y)|\right| \\ & \leq C_n \frac{\delta r(B)}{|z_0-y|^{n+1}} . \end{aligned}. \end{equation}
Putting ($\ref{D1}$) and ($\ref{D2}$) together gives: 
\begin{equation}\label{S1}\begin{aligned} S_1 & \leq \int_{B \cap \pi^{-1}(A_{L}(z_0,r_0,r))} \frac{c \delta r(B)}{|z_0-y|^{n+1}} d\mu(y) \\ & \leq c \int_{\frac{1}{2} r_0 \leq |y-z_0| \leq 2r} \frac{\delta r(B)}{|z_0-y|^{n+1}} d\mu(y) \\
& \leq c\delta r(B) \int_{\frac{1}{{2r}^{n+1}}}^{(\frac{2}{r_{0}})^{n+1}} \mu\left( \left\lbrace y, \frac{1}{|z_0 - y|^{n+1}}>t \right\rbrace \right) dt, \\
&\leq c\delta r(B) \int_{\frac{1}{{2r}^{n+1}}}^{(\frac{2}{r_{0}})^{n+1}} \mu(B(z_0,t^{-\frac{1}{n+1}})) dt, \\ & \leq c\delta r(B) \int_{\frac{1}{{2r}^{n+1}}}^{(\frac{2}{r_{0}})^{n+1}} \frac{1}{t^{\frac{n}{n+1}}} dt  \\ &\leq \frac{C' \delta r(B)}{r_0} \leq 1\mbox{  by choosing } {\delta}^{\frac{1}{2}} < {C'}^{-1} . \end{aligned}\end{equation}
 
Estimating $S_2$ from ($\ref{S1S2}$):
 
 $$\begin{aligned} S_2 &= \left| \int_{B \cap \pi^{-1}(A_{L}(z_0,r_0,r))} K(z_1-y) d\mu(y) - \int_{A(z_1, r_0, r)} K(z_1-y) d\mu(y) \right| \\ &\leq \int_{B \cap ( \pi^{-1}(A_{L}(z_0,r_0,r)) \triangle A(z_1, r_0,r))} |K(z_1-y)| d\mu(y)\end{aligned}$$
 
Now, we claim that for $\delta>0$ small enough, we have: \begin{equation} \label{annulus} \Sigma \cap B \cap (\pi^{-1}(A_{L}(z_0,r_0,r)) \triangle A(z_1,r_0,r)) \subset A(z_1, \frac{1}{2} r_0, 2r_0) \cup A(z_1,\frac{1}{2}r,2r). 
\end{equation}
 We will only treat the case where $ \pi(y) \in A(z_0, r_0, r)$ and $y \notin A(z_1, r_0,r)$. The other case follows in exactly the same manner.
First, note that in the above case, either $|y-z_1| \leq r_0 $, implying in particular that $ |y-z_1| \leq 2r_0$.
Moreover, for such a $y$, 
\begin{align*}
|y-z_1| &\geq |\pi(y)-z_0| - |y- \pi(y)| - | z_1-z_0|, \\ & \geq r_0 - 2 \delta r(B) , \mbox{ by } \eqref{flatdelta}, \\
& \geq \frac{1}{2} r_0, 
\end{align*}  and hence, $y \in A(z_1, \frac{1}{2} r_0 , 2r_0)$.
Otherwise, $|y-z_1| > r$ (a fortiori, $|y-z_1| > \frac{1}{2} r$) and \begin{align*} |y-z_1| &\leq |y- \pi(y)| + |\pi(y)-z_0|+|z_1-z_0|, \\
& \leq 2 \delta r(B) + r \leq 2r \end{align*} implying $y \in A(z_1, \frac{1}{2} r, 2r)$.
Hence, $$ \Sigma \cap B \cap( \pi^{-1}(A_{L}(z_0, r_0, r))\cap {A(z_1, r_0,r)}^C ) \subset A(z_1, \frac{1}{2} r_0, 2r_0) \cup A(z_1,\frac{1}{2}r,2r).$$ 

Using $\eqref{annulus}$, we obtain 
 \begin{equation}\label{S2}\begin{aligned} S_2 & \leq \int_{ A(z_1, \frac{1}{2} r_0, 2r_0)} \frac{1}{|z_1 - y|^{n}} d\mu(y) + \int_{ A(z_1,\frac{1}{2}r,2r)}  \frac{1}{|z_1 - y|^{n}} d\mu(y) \\ & \leq 2^n \frac{ \mu(B(z_1, 2r_0))}{r_0^n} + 2^n \frac{\mu(B(z_1, 2r))}{r^n} \\ &\leq C_n. \end{aligned} \end{equation} 

The estimates ($\ref{S1}$) and ($\ref{S2}$) combined give:

\begin{equation} \label{A} \left| \frac{x-z_0}{r_0} \cdot (R_{r_0,r} \nu (z_0) - R_{r_0,r} \mu(z_1)) \right|  \leq |\frac{x-z_0}{r_0}| \cdot C \leq C'. \end{equation}

 Let us now estimate ($\ref{Beq}$):
first note that \begin{align*} | (x-z_0)-(x_1-z_1)|  &\leq |x-x_1| + |z_0-z_1|, \\ &\leq 2 \delta r(B) , \mbox{ by } \eqref{flatdelta}.
\end{align*}

Moreover,\begin{align*}
 |R_{r_0,r} \mu(z_1)| &\leq \int_{r_0 < |y-z_1| \leq r} \frac{1}{|y-z_1|^{n}} d\mu(y) \\ & \leq  \int_{r^{-n}}^{r_0^{-n}} \mu(\{y: |K(y-z_1)|> t^{-n}\})dt \\ & \leq  \int_{r^{-n}}^{r_0^{-n}} \frac{1}{t} dt \\ & \leq C \log\left(\frac{r(B)}{r_0}\right) \\ &\leq C' |\log(\delta)|. \end{align*}
Therefore, assuming $\delta$ is small enough:
 \begin{align}\label{B} \left| \frac{(x-z_0)-(x_1-z_1)}{r_0} \cdot R_{r_0,r} \mu(z_0) \right| & \leq C \frac{\delta r(B)}{r_0} |\log{\delta}| \nonumber \\ & \leq C \delta^{\frac{1}{2}} |\log{\delta}| \leq C \end{align}

Finally, we estimate ($\ref{Ceq}$):
we want to apply Lemma $\ref{smallRiesz}$ to evaluate $\left| \frac{x_1-z_1}{r_0} \cdot R_{r_0,r} \mu(z_1) \right|$. But we do not have $x_1 \in B(z_1,r_0)$. Nevertheless, we have:
\begin{equation}\label{r0}|x_1-z_1| \leq |x_0-z_0|+|z_0-z_1|+|x_0-x_1| \leq |x_0-z_0|+2 \delta r(B) \leq 2r_0.\end{equation}
Using ($\ref{r0}$), and applying Lemma $\ref{smallRiesz}$ to the first term, we have:
\begin{align}\label{almostsmall} \left| \frac{x_1-z_1}{r_0} \cdot R_{r_0,r} \mu(z_1) \right| & \leq 2 \left| \frac{x_1-z_1}{2r_0} \cdot  R_{2r_0,r} \mu(z_1) \right| + \left|\frac{x_1-z_1}{r_0} \cdot R_{r_0,2r_0} \mu(z_1) \right|\nonumber \\& \leq 2c + \left| \frac{x_1-z_1}{r_0} \right| \cdot \left|R_{r_0,2r_0}\mu(z_1) \right|. \end{align} 
To estimate the second term on the right hand side of the inequality in $\eqref{almostsmall}$, we simply notice that: \begin{equation} \label{C}|R_{r_0,2r_0} \mu(z_1)| \leq r_0^{-n} \mu(B(z_1,2r_0)) \leq \tilde{c}.\end{equation}
This implies the uniform boundedness of $C$.

Combining our estimates in ($\ref{A}$), ($\ref{B}$), ($\ref{almostsmall}$), and ($\ref{C}$), we have proven the claim we had set out to prove: namely, that for all $z_0 \in \frac{1}{2}B \cap \widetilde{\Sigma}$, and $r_0, r$ with $\delta^{\frac{1}{2}}r(B) \leq r_0 \leq r \leq r(B)$, if $\delta$ is small enough, 
\begin{equation}\label{estimate2} \left| \frac{x-z_0}{r_0} \cdot R_{r_0,r}\nu(z_0) \right| \leq c, \quad \mbox{for} \; x\in \widetilde{\Sigma} \cap B(z_0,r_0).
\end{equation} 

  \end{proof}     
  \begin{theorem}\label{bigballs}
  Let $\mu$ be a uniformly distributed measure with $n=dim_{0}\mu$, $K$ a compact set. For every $\epsilon>0$ , there exists some $\tau>0$ such that every ball $B$ centered in $\Sigma$ and contained in $K$, contains another ball $B'$ also centered in $\Sigma$, which satisfies $\beta_{\mu}^{n}(B') \leq \epsilon$, and $r(B') \geq \tau r(B)$. Moreover, $\tau$ only depends on $\epsilon$, $K$, $n$ and $d$.
  \end{theorem}
  \begin{proof}
  We just apply Lemma $\ref{inductflat}$ $(d-n)$ times. Indeed, since $\beta_{\mu}^{(d)}(B)=0$, there exists $B_1 \subset B$, $\beta_{\mu}^{(d-1)}(B_1) \leq \epsilon_1$, $r(B_1) \sim r(B)$. By induction, we get a ball $B_{d-n} \subset B$, $r(B_{d-n}) \sim r(B)$, $\beta_{\mu}^{(n)}(B_{d-n}) \leq \epsilon_{d-n}$. Making the successive $\epsilon_i$'s as small as needed, since there are only finitely many steps, letting $\epsilon=\epsilon_{d-n}$, and $B'=B_{d-n}$, we get $\epsilon >0$, $B' \subset B$, $r(B') \sim r(B)$ such that: $\beta_{\mu}^{(n)}(B') \leq \epsilon$. \end{proof}
 
\newpage
 \section{Stability of the $\beta$-numbers}
In the following section, we will write $\beta(B)$ for $\beta^{n}(B)$.

 \begin{lemma} \label{betasupport}
  Let $\mu$ be a Radon measure, $\Sigma$ its support. Let $\mu_{x,r}$  be the following measure (and $\Sigma_{x,r}$ its support):
$$ \mu_{x,r}(A)=\mu(rA+x).$$
Then we have:
 \begin{equation}\label{trans} b\beta_{\mu_{x,r}}(B)=b\beta_{\mu}(rB+x) 
 \end{equation}
 \end{lemma}
 \begin{proof}
 
 It is easily seen that: $\Sigma=r\Sigma_{x,r}+x$.
 Let $L$ be an $n$-plane.
 
  Let $y \in \Sigma_{x,r} \cap B$. Then $y'=ry+x$ is in $\Sigma \cap B_{x,r}$ and $\mbox{dist}(y,L)=\frac{1}{r} \mbox{dist}(y',L+x)$. Hence, \begin{equation}\label{sigma1} \sup_{\Sigma_{x,r} \cap B} \mbox{dist}(y,L)= \frac{1}{r} \sup_{\Sigma \cap B_{x,r}} \mbox{dist} (y',L+x). \end{equation}
  
  On the other hand, let $p \in L \cap B$. Then $\mbox{dist}(p,\Sigma_{x,r})=\frac{1}{r}\mbox{dist}(rp+x,\Sigma)$, where $rp+x=p'$, $p' \in (L+x)\cap B_{x,r}$. Hence,
  \begin{equation}\label{sigma2} \sup_{L \cap B} \mbox{dist}(p,\Sigma_{x,r})=\sup_{L+x \cap B_{x,r}} \mbox{dist}(p',\Sigma). \end{equation}
  
  Adding ($\ref{sigma1}$) and ($\ref{sigma2}$), and taking the infimum over all $n$-planes proves ($\ref{trans}$).
 \end{proof}
 
 We can now prove the following theorem. It states that the flatness of $\mu$ on a fixed number of bigger scales than $B$ implies flatness at scale $B$. 
 
 \begin{theorem}\label{stability1}
 Let $\mu$ be a Radon measure on $\mathbb{R}^d$ that is doubling and $n$-asymptotically optimally doubling, and $K$ a compact set in $\mathbb{R}^d$.
      Let $\epsilon >0$, and $\delta_{0}$ be $\tau_0$ from Theorem $\ref{flat}$.                         
\\ There exists an integer $N >0$, depending only on the measure $\mu$, $\epsilon$,  $d$ and $K$ such that for every ball $B$ centered in $\Sigma \cap K$, if
$$2^N B \subset K, \; \beta_{\mu} (2^{k} B) \leq \frac{\delta_{0}}{4},\; 1 \leq k \leq N $$
then 
$$b\beta_{\mu}(B) \leq \epsilon .$$
\end{theorem}
\begin{proof}
We argue by contradiction. Suppose there is no such $N$. Then, for every $j$, there exists a ball $B_j=B(x_j, r_j)$, $x_j \in K \cap \Sigma$, $2^j r_j \leq \mbox{diam}(K)$ such that: 
\begin{equation}\label{contra}\beta_{\mu} (2^{k} B_j) \leq \frac{\delta_{0}}{4}, 1 \leq k \leq j$$ but $$b\beta_{\mu}(B_j) \geq \epsilon \end{equation}
Note that $x_j \in \Sigma \cap K$, $2^j r_j \leq \mbox{diam}(K)$ imply that, passing to a subsequence, $r_j \to 0$ and $x_j \to x$, $x \in K$, as $j \to \infty$.
Now, let $\mu_j$ be the measure defined as: $$\mu_j(A) = \frac{\mu(r_j A + x_j)}{\mu(B_j)}.$$

There exists some subsequence of $\mu_j$ that converges weakly to a measure $\nu$ as $j \to \infty$. Indeed, for any ball $B(0,R)$, if $C$ is the doubling constant of $\mu$, and $t(R)=\frac{\log(R)}{\log(2)}$, then: $$ \mu_j(B(0,R)) = \frac{\mu(B(x_j,Rr_j)}{\mu(B(x_j,r_j)} \leq C^{t(R)}. $$
Therefore, $\sup_{j}(\mu_j(B(0,R))) \leq C^{t(R)}$, for every $R>0$.

Since $x_j$ converges to $x$ and $r_j$ to $0$, $\nu$ is a pseudo-tangent measure of $\mu$ at $x$, and is therefore $n$-uniform by Lemma $\ref{pseudounif}$ since $\mu$ is asymptotically doubling by hypothesis.
Moreover, since the $\mu_j$'s are doubling with the same constant, using Lemma $\ref{weakconv2}$ and $\eqref{contra}$:
$$ 2 b\beta_{\nu}(B(0,2))  \geq \limsup_{j \to \infty} b\beta{\mu_{j}}(B(0,1)) = \limsup_{j \to \infty} b \beta_{\mu}(B_j) \geq \epsilon. $$
On the other hand, for all $k \geq 0$, by ($\ref{contra}$), 
$$  \beta_{\nu}(B(0,2^{k-1})) \leq 2\liminf_{j \to \infty} \beta_{\mu_j}(B(0,2^{k})) = 2\liminf_{j \to \infty} \beta_{\mu}(2^{k} B_{j}) \leq \frac{\delta_{0}}{2}. $$

Let $\lambda$ be the tangent measure of $\nu$ at $\infty$. Define $\nu_k$ in the following manner:
$$ \nu_k(A)=\frac{\nu(2^{k} A)}{2^{nk}}.$$
Then $\nu_{k} \rightharpoonup \lambda$ and:
$$\beta_{\lambda}(B(0,1)) \leq 2\liminf_{k \to \infty}\beta_{\nu_{k}}(B(0,2))= 2\liminf_{k \to \infty} \beta_{\nu}(B(0,2^{k+1}) \leq \delta_{0}.$$
By Theorem $\ref{flat}$, this implies that $\nu$ is flat, contradicting $b\beta_{\nu}(B(0,1)) >\epsilon$. 
\end{proof}

\begin{corollary}\label{stability2}
Let $\eta > 0$,  $K$ compact set in $\mathbb{R}^d$, $\mu$ a measure that is doubling and asymptotically optimally doubling.
There is a constant  $\delta_{1} > 0$ satisfying the following property. If 
\begin{equation}
k \geq 0, \ B(x_0,r) \subset K, \ x_0 \in \Sigma \cap K, b\beta_{\mu}(B(x_0,r)) \leq \delta_{1} ,
\end{equation}
then:
$$b\beta_{\mu}(2^{-k}B(x_0,r))) \leq \eta . $$
\end{corollary}
\begin{proof}
Let $\delta_0$ be as in Lemma $\ref{stability1}$, $\epsilon_0=\min(\frac{\delta_0}{4},\eta)$. 
Consider $N=N(\epsilon_0)$ from Lemma $\ref{stability1}$.
Denote $B(x_0,r)$ by $B$.
Assume $k \geq N$. We prove by induction that: for $0 \leq j \leq k$, $B_j= 2^{-j}B$, $b\beta_{\mu}(B_j) \leq \epsilon_0$.
\\If $0 \leq j \leq N$, $$ b\beta_{\mu}(B_j) \leq \frac{r(B)}{r(B_{j})}b\beta_{\mu}(B) \leq 2^{N} \delta_{1} \leq \epsilon_{0},$$ if we assume $\delta_1\leq 2^{-N} \epsilon_{0}$.
\\If $j>N$, $b\beta_{\mu} (B_{j-k}) \leq \epsilon_{0} \leq \frac{\delta_{0}}{4}, k=1, \ldots, N$ implies by Lemma $\ref{stability1}$ that $b\beta_{\mu}(B_{j}) \leq \epsilon_0$. This ends the proof. \end{proof}

We can reformulate Theorem $\ref{stability1}$ and Corollary $\ref{stability2}$ in terms of dyadic cubes.

\begin{corollary}\label{stability1cubes}
Let $\mu$ be a Radon measure on $\RR^{d}$ that is doubling, $n$-asymptotically doubling, and locally Ahlfors-regular. Let $K$be  a compact set, and $\mathcal{D}^{\mu}$ the dyadic decomposition of $\mu$ described in Corollary $\ref{Dyadic3}$. Let $\epsilon>0$ and $\delta_0$ be $\tau_0$ from Theorem $\ref{flat}$.
There exists an integer $N>0$ depending only on $\epsilon$, $d$ and $K$ such that for every $R \in \mathcal{D}^{\mu}$, $R \subset K$, if:
$$ R^{(N)} \subset K, \mbox{ } \beta_{\mu}(R^{(k)}) \leq \frac{\delta_0}{4},\mbox{ } 1\leq k \leq N$$
where $R^{(k)}$ denotes the ancestor of $R$ of generation $k$, then $$b\beta_{\mu}(R) \leq \epsilon.$$
\end{corollary}
\begin{proof}
We argue by contradiction. Suppose there is no such $N$. Then, for every $j$, if we denote the center $z_{R_j}$ by $z_j$ and $l(R_j)$ by $l_j$, there exists a dyadic cube $R_{j}$ such that  $z_j \in K$, $2^{j} l_j \leq diam(K)$, $R_j^{(j)} \subset K$, 
 and 
$$\beta_{\mu}(R_{j}^{(k)}) \leq \frac{\delta_0}{4}, \mbox{ } 1 \leq k \leq j$$
but $$b\beta(R_j) \geq \epsilon.$$
By compactness, there exists $z \in supp(\mu)$ such that $z_j \to z$ (without loss of generality by passing to a subsequence). Moreover, $l_j \to 0$.
Now let $\mu_j$ be the measure defined as:
\begin{equation}
\mu_{j}(A)= \frac{\mu(l_j A +z_j)}{\mu(B(z_j,3l_j))}.
\end{equation}

The rest of the proof follows in exactly the same manner as Theorem $\ref{stability1}$ since  $\beta_{\mu}(R_j)= \beta_{\mu}(B(z_j, 3l_j))$ by definition.
\end{proof}

\begin{corollary}\label{stability2cubes}
Let $\eta >0$, $K$ compact set in $\RR^{d}$, $\mu$ a measure that is doubling, asymptotically optimally doubling and locally Ahlfors-regular. There is a constant $\delta_1>0$ satisfying the following property. If 
\begin{equation}
k \geq 0 \;  R \in \mathcal{D}^{\mu} \; R \subset K \; \beta_{\mu}(R) \leq \delta_1, 
\end{equation}
then
\begin{equation}
b\beta_{\mu}(Q) \leq \eta , \mbox{ for all children } Q \mbox{ of } \mathcal{R}.
\end{equation}
\end{corollary}
\begin{proof}
The proof follows from Corollary $\ref{stability1cubes}$ in the same manner that the proof of Corollary $\ref{stability2}$ follows from Theorem $\ref{stability1}$.
\end{proof}
We prove that a uniformly distributed measure is doubling and asymptotically optimally doubling so that we can apply Theorem $\ref{stability1}$ and Corollary $\ref{stability2}$ on one hand, and obtain a dyadic decomposition by using Theorem $\ref{dyadic}$ on the other hand.

\begin{lemma}\label{doubling}
Suppose $\mu$ is a uniformly distributed Radon measure in $\mathbb{R}^d$ with $dim_0\mu=n$, $dim_\infty \mu =p$ and such that: $\mu(B(x,r))=\phi(r)$, for $x \in \Sigma$.
Then $\mu$ is doubling and asymptotically optimally $n$-doubling.
\end{lemma}
\begin{proof}
We first prove that $\mu$ is doubling. This follows easily from Theorem $\ref{doublingKiP}$. Indeed, if $x \in supp(\mu)$, $r>0$
$$ \mu(B(x,2r)) \leq 10^{d} \mu(B(x,r)).$$

To prove that $\mu$ is $n$-asymptotically optimally doubling, let $K$ be a compact set such that $K \cap \Sigma \neq \emptyset $, and $\tau \in (0,1)$. Choose $x$ in $K \cap \Sigma$.
 Then $$\frac{\mu(B(x,\tau r))}{\mu(B(x,r))} = \frac{\phi(\tau r)}{\phi(r)}$$
 and hence, using Theorem $\ref{dim}$, $$\begin{aligned} \lim_{r \to 0} \sup_{x \in K \cap \Sigma} \frac{\mu(B(x,\tau r))}{\mu(B(x,r))} &= \lim_{r \to 0} \inf_{x \in K \cap \Sigma} \frac{\mu(B(x,\tau r))}{\mu(B(x,r))} \\ &=\lim_{r \to 0} \frac{\phi(\tau r)}{\phi(r)}\\&=\lim_{r \to 0} \frac{\phi(\tau r)}{\tau^n r^n} \frac{r^n}{\phi(r)} \tau^{n}\\&= \tau ^ {n}. \end{aligned}$$ \end{proof}
 
\begin{theorem}\label{stabilitybigunif} Let $\mu$ be a uniformly distributed measure on $\mathbb{R}^d$, $K$ compact set in $\mathbb{R}^d$. Fix $\eta >0$. Then there exists a constant $c >0$, depending on $n$, $d$, $K$ and $\eta$ with the following property:
for every ball $B$ centered in $\Sigma \cap K$, there exists a ball $B'$ centered in $\Sigma \cap K$, $B' \subset B$, and $r(B') \geq c r(B)$ such that for every $k \geq 0$, $$b\beta(2^{-k} B') \leq \eta. $$
\end{theorem}
\begin{proof} 
Let $\delta_1$ be the constant from Corollary $\ref{stability2}$ corresponding to $\eta$, $\delta_1=\delta_1(\eta)$.
Let $\delta_0$ and $N$ be the constants from Theorem $\ref{stability1}$ corresponding to $\delta_1$, and let $B''$ be the ball given by Theorem $\ref{bigballs}$ and $\tau$ the constant given by the same theorem, corresponding to $2^{-N} \delta_0$, i.e. $\tau=\tau(2^{-N} \delta_0)$.

Then we have $B'' \subset B$, $r(B'') \geq \tau r(B) $ and $\beta(B'') \leq 2^{-N} \delta_{0}$. Hence, for $0 \leq k \leq N$, $\beta(2^{-k} B'') \leq \frac{r(B'')}{2^{-k}r(B'')}\beta(B'') \leq \delta_0$. Let $B'=2^{-N} B''$.
We have: $\beta(2^{k}B') \leq \delta_0$, for all $1 \leq k \leq N$.
By Theorem $\ref{stability1}$, this implies that $b\beta(B') \leq \delta_1$, which in turn, by Corollary $\ref{stability2}$, gives $b\beta(2^{-k} B') \leq \eta$ for each $k \geq 0$. Letting $c= 2^{-N} \tau$, we end the proof. Note that $c$ does not depend on our choice of $B$. \end{proof}
Using Theorem $\ref{stabilitybigunif}$ iteratively, we easily obtain the following corollary.

\begin{corollary}\label{protoBPLG}
Let $\mu$ be a uniformly distributed measure on $\mathbb{R}^d$, $K$ compact set in $\mathbb{R}^d$. Fix $\eta >0$. Then there exists $\delta>0$ depending on $\eta$, $\mu$, $n$ and $d$ such that if $B$ is a ball centered in $\Sigma \cap K$, with $\beta_{\mu}(B) < \delta$, then $b\beta_{\mu}(B') \leq \eta$ for any ball $B' \subset \frac{1}{2} B$ centered in $\Sigma $.
\end{corollary}

We can finally prove Theorem $\ref{mainunif}$. The proof is identical to the proof of Theorem $[1.2]$ in $\cite{T}$. We repeat it for the reader's convenience.
\begin{proof}
Let $K$ be compact, $\epsilon>0$. By Theorems $\ref{bigballs}$ and $\ref{protoBPLG}$, there exists $c$ depending on $\epsilon$ and $K$ such that any ball centered in $\Sigma \cap K$ contains a ball $B'$, $r(B') \geq c r(B)$ contained in $B$ such that every ball $B"$ centered in $\Sigma \cap K$ and contained in $\frac{1}{2}B'$ satisfies $b\beta_{\mu}(B") <\epsilon$.
An application of Theorem $[15.2]$ from $\cite{DT}$ gives the result.
\end{proof}

\end{document}